\documentclass[12pt,reqno]{article}
\usepackage{geometry}
\usepackage{amsmath,amsthm,amssymb,mathtools}
\usepackage{graphicx} 
\usepackage{microtype}
\usepackage{hyperref}
\usepackage[capitalize]{cleveref}
\usepackage[hyperref,doi,url=false,eprint=false,style=alphabetic,maxbibnames=99]{biblatex}
\usepackage{enumitem}
\usepackage{caption}
\usepackage{subcaption}
\usepackage{color}

\AtEveryBibitem{
	\clearlist{address}
	\clearfield{date}
	\clearfield{isbn}
	\clearfield{issn}
	\clearlist{location}
	\clearfield{month}
	\clearfield{series}
	\clearfield{note}

	\ifentrytype{book}{}{
	\clearlist{publisher}
	\clearname{editor}
	}
}

\title{Veering triangulations and transverse foliations}
\author{Jonathan Zung\footnote{jzung@mit.edu}}
\date{}

\newcommand{\Mpunc}{M^\circ}
\newcommand{\Mbar}{\overline{M^\circ}}

\newcommand{\phipunc}{\phi^\circ}
\newcommand{\phibar}{\overline{\phi^\circ}}

\newcommand{\RR}{\mathcal R}
\newcommand{\R}{\mathbb R}
\newcommand{\lR}{\pi^{-1}(\RR)}
\newcommand{\F}{\mathcal F}
\newcommand{\skel}{\tau^{(2)}}
\newcommand{\ospace}{\mathcal O}

\let\temp\phi
\let\phi\varphi
\let\varphi\temp

\newcommand{\lift}[1]{\widetilde{#1}}
\DeclareMathOperator{\Homeo}{Homeo}

\newtheorem{thm}{Theorem}[section]

\newtheorem{mainthm}{Theorem}

\newtheorem{lem}[thm]{Lemma}
\newtheorem{prop}[thm]{Proposition}

\newtheorem{question}[thm]{Question}

\theoremstyle{definition}

\theoremstyle{remark}
\newtheorem{rmk}[thm]{Remark}

\begin{document}

\maketitle

\begin{abstract}
We present a combinatorial approach to the existence of foliations and contact structures transverse to a given pseudo-Anosov flow. Let $\phi$ be a transitive pseudo-Anosov flow on a closed oriented 3-manifold. Our main technical result is that every codimension 1 foliation transverse to $\phi$ is carried by a single branched surface coming from a veering triangulation. Combined with recent breakthrough work of Massoni, this reduces the existence problem for transverse foliations to something like the feasibility of a system of inequalities (rather than equations!) over $\Homeo_+([0,1])$. As a proof of concept, we show that for the hyperbolic, fibered, non-L-space knot $10_{145}$, the natural pseudo-Anosov flow on the slope $s$ Dehn surgery admits a transverse foliation for $s\in (-\infty, 3)$, but does not admit such a foliation for $s\in [5,\infty)$. The negative result is part of a more general Milnor--Wood type phenomenon which puts limitations on some well known methods for constructing taut foliations on Dehn surgeries.
\end{abstract}

\section{Introduction}
A major goal of the theory of taut foliations is to construct a pseudo-Anosov flow transverse (or almost-transverse) to a given taut foliation on a hyperbolic 3-manifold. When possible, this deepens our understanding of the asymptotic geometry of both the foliation and the flow. In the most harmonious picture, one may intertwine three different actions of $\pi_1(M)$: the action on an ideal boundary of the orbit space of the flow, the action on a universal circle for the foliation, and the action on the sphere at infinity of $\mathbb H^3$. The requisite pseudo-Anosov flow has been constructed in many cases: slithering foliations as shown by Thurston \cite{thurston.ThreemanifoldsFoliationsCircles}, more generally $\R$-covered foliations or even foliations with 1-sided branching due to Fenley \cite{fenley.RegulatingFlowsTopology} and Calegari \cite{calegari.GeometryRcoveredFoliations}, and also finite depth foliations by work of Gabai and Mosher \cite{mosher.LaminationsFlowsTransverse,landry.tsang.EndperiodicMapsSplitting}. Calegari made progress in the general setting, but the problem remains open in general \cite{calegari.ProblemsFoliationsLaminations}.

If one believes in this program, then a route towards classifying taut foliations on a 3-manifold is first to classify transitive pseudo-Anosov flows (which are essentially combinatorial in nature), and then to classify foliations which are transverse to a given flow. We study the second problem in this paper.

\begin{question}\label{ques:th1}
    Given a transitive pseudo-Anosov flow on a closed, oriented 3-manifold, when does there exist a foliation transverse to the flow? How are these various foliations related?
\end{question}
\begin{question}\label{ques:th2}
    Given a transitive pseudo-Anosov flow on a closed, oriented 3-manifold, when does there exist a contact structure transverse to the flow?
\end{question}

Remark that contact structures and foliations transverse to a transitive pseudo-Anosov flow (without 1-prong singularities) are automatically tight or taut. For contact structures this is discussed in \cite[Proposition 5.1]{zung.PseudoAnosovRepresentativesStable}, while for foliations this is a consequence of transitivity of the flow. Thurston challenges his readers with both \cref{ques:th1} and \cref{ques:th2} in his original paper on slitherings, and suggests that there ought to be a connection with the Milnor--Wood inequality and rotation numbers \cite[Section 8]{thurston.ThreemanifoldsFoliationsCircles}.  In this paper, we develop some combinatorial tools to address these questions.

Throughout, $\phi$ is a transitive pseudo-Anosov flow (henceforth possibly with 1-prong singularities) on a closed, oriented 3-manifold $M$. One can always remove finitely many periodic orbits $\gamma_1,\dots, \gamma_n$ so that the restriction of $\phi$ to $\Mpunc:=M\setminus \{\gamma_1,\dots,\gamma_n\}$ satisfies the technical condition ``no perfect fits''. Agol and Gueritaud constructed a canonical \emph{veering} ideal triangulation $\tau$ of $\Mpunc$ which encodes the stable and unstable laminations of $\phi$. The 2-skeleton of $\tau$, denoted by $\skel$, is naturally a branched surface transverse to $\phipunc$. On each cusp of $\Mpunc$, call the meridional slope 0 and the degeneracy slope $\infty$. We say that there is a  \emph{positive solution to the holonomy problem} if $\skel$ carries a lamination having positive boundary slope at every cusp of $\Mpunc$. Similarly, we say that there is a negative solution to the holonomy problem if $\skel$ carries a lamination having negative slope at every cusp. We emphasize that the boundary need not have a linear foliation; it simply needs to have a foliation with positive or negative slope. The branched surface $\skel$ has no triple points, so a lamination carried by $\skel$ is completely specified by a choice of an element of $\Homeo^+([0,1])$ at each edge of $\tau$. This is the sense in which the holonomy problem is a system of inequalities over $\Homeo^+([0,1])$.

\begin{mainthm}\label{thm:converse}
There is a foliation transverse to $\phi$ if and only if there is a positive and a negative solution to the holonomy problem.
\end{mainthm}

The final ``only if'' statement uses the Eliashberg--Thurston theorem to produce positive and negative contact structures approximating a given foliation. The ``if'' statement uses Massoni's recently announced converse which produces a $C^0$ foliation given a pair of tight contact structures and a common transverse vector field \cite{massoni.TautFoliationsContact}.

Here is the technical input to \cref{thm:converse}:
\begin{prop}\label{prop:carried}
    Let $\F$ be a foliation transverse to $\phi$. Then the restriction of $\F$ to $\Mpunc$ splits open to a lamination fully carried by $\skel$.
\end{prop}

This result follows the spirit of a line of work by Landry, Minksy, and Taylor. They previously proved that transverse \emph{surfaces} are carried by $\skel$ \cite{landry.VeeringTriangulationsThurston, landry.minsky.ea.TransverseSurfacesPseudoAnosov}. Their proof uses a somewhat different strategy based on the notion of a taut surface and its Euler characteristic. Their strategy might work for foliations with an invariant transverse measure of full support, but does not seem to apply to our situation in generality.

\subsection{Obstructing transverse taut foliations}
The L-space conjecture suggests that for a non-L-space knot in $S^3$, every nontrivial Dehn surgery admits a taut foliation. The most successful systematic techniques for constructing such taut foliations in fact produce foliations transverse to a natural pseudo-Anosov flow. Roberts showed that for fibered knots in $S^3$, surgery slopes in the interval $(-1,1)$ admit taut foliations \cite{roberts.TautFoliationsPunctured}. Roberts and Li adapted this construction to non-fibered knots \cite{li.roberts.TautFoliationsKnot}. Krishna pushed these constructions to a much larger interval of slopes in the case of positive braid knots, and Zhao did the same for a special class of fibered knots \cite{krishna.TautFoliationsBraid,krishna.TautFoliationsPositive,zhao.CoorientableTautFoliations}. All of these techniques proceed by constructing a branched surface and analyzing the carried laminations. The branched surface in question is always transverse (or almost-transverse) to a pseudo-Anosov flow in the knot complement, and the resulting foliation is transverse (or almost-transverse) to Fried surgery on the pseudo-Anosov flow. Roberts showed that in some cases (eg fibered hyperbolic knots with fractional Dehn twist coefficient 0), every surgery slope in $(-\infty,\infty)$ results in a pseudo-Anosov flow with a transverse foliation.  However, experts suspect that these techniques will not produce taut foliations on all non-L-space surgeries. We give definite proof of these limitations:

\begin{mainthm}\label{thm:nofol1}
    Let $K$ be the hyperbolic, fibered, genus 2, non-L-space knot $10_{145}$. For $s\in (-\infty,3)$, slope $s$ surgery on $K$ admits a taut foliation transverse to the natural pseudo-Anosov flow on $S^3_s(K)$. For $s\in [5,\infty)$, there does not exist a foliation transverse to this flow.
\end{mainthm}
Note that this is not a potential counterexample to the L-space conjecture, as Dunfield has found taut foliations on all non-trivial surgeries on prime non-L-space knots up to 16 crossings \cite[Theorem 8.3]{dunfield.FloerHomologyGroup}. The bound $5$ is not tight, and could likely be improved with more careful analysis. One may be able to prove similar non-existence results using Heegaard Floer homology, since the existence of a transverse foliation implies the nontriviality of $HF_{red}$ in the grading corresponding the orthogonal 2-plane field to the flow. In forthcoming work, Baldwin--Hedden--Krishna provide a Heegaard Floer homology obstruction to the existence of transverse taut foliations in certain Dehn surgeries along fibered strongly quasipositive knots (including $10_{145}$) \cite{krishna.baldwin.hedden}. Note that we have not ruled out any slope $s < 2g(K)-1$, and we plan to investigate this problem in the future.

\Cref{thm:nofol1} is a specific application of a more general criterion:
\begin{mainthm}\label{thm:nofol2}
    Suppose $\phi$ is a pseudo-Anosov flow without perfect fits having a single singular $k$-prong orbit such that the first return map rotates by $\ell$ prongs and $gcd(k,\ell)=1$. Then there is an open interval of slopes containing the degeneracy slope such that Fried surgery along any slope in this interval not equal to the degeneracy slope results in a pseudo-Anosov flow with no transverse foliation. 
\end{mainthm}
This theorem is widely applicable: it applies to 31,138 of the 87,047 cusped 3-manifolds in the census of veering triangulations up to 16 tetrahedra \cite{parlak.schleimer.ea.Veering03Code}. This includes many examples which are not fibered. The condition on the first return map is necessary in light of Roberts' result on fibered knots with fractional Dehn twist coefficient zero. We view \cref{thm:nofol2} as the simplest instance of a Milnor--Wood type inequality for pseudo-Anosov flows. We will expand upon this analogy in the next subsection.



\begin{subsection}{Analogies with the Seifert fibered case}
When studying homeomorphisms of surfaces, one should start with the periodic homeomorphisms. Similarly, when studying flows on 3-manifolds, one should start with Seifert fibrations. Let us begin by recounting the now-well-understood story of foliations transverse to Seifert fibrations. The first theorem of Brittenham, generalizing a theorem of Thurston, shows that the most interesting taut foliations are transverse to the Seifert fibration:

\begin{thm}[Thurston, Brittenham \cite{brittenham.EssentialLaminationsSeifertfibered}]
Let $E\xrightarrow{f} \Sigma$ be a Seifert fibration over an orbifold $\Sigma$, and let $\F$ be a taut foliation of $E$. Then there is an isotopy of $\F$ so that after the
isotopy, every leaf of $\F$ is either vertical (saturated by Seifert fibers) or horizontal (transverse to Seifert fibers). When $\F$ has vertical leaves, one can decompose $E$ along the (possibly non-compact) vertical leaves to obtain simpler Seifert fibered pieces admitting horizontal foliations.
\end{thm}

Transverse foliations are classified by representations of $\pi_1(\Sigma)$ into $\Homeo_+(S^1)$. When $\Sigma$ has no orbifold points, the Milnor--Wood inequality dictates whether a horizontal foliation exists. Let's review how this works. Suppose $\rho:\pi_1(\Sigma) \to \Homeo_+(S^1)$ is a representation, and $\F$ is the corresponding transverse foliation. Choose a basepoint $x\in \Sigma$ and let $\gamma_1,\dots, \gamma_{2g}$ be a collection of disjoint embedded loops based at $x$ whose complement is a disk $D^2$. Cut open the 3-manifold along $f^{-1}(\cup_i \gamma_i)$ to obtain a trivially foliated $D^2 \times S^1$. Then $\F$ induces a foliation on $D^2\times S^1$. Each arc $\gamma_i$ contributes two regions $I\times S^1$ in $\partial D^2 \times S^1$.These two regions have horizontal foliations whose holonomies are $\rho(\gamma_i)$ and $\rho^{-1}(\gamma_i)$ respectively. In order for the foliation to extend across $D^2 \times S^1$, the product of commutators spelled out by $\partial D^2$ must be the identity and have rotation number equal to the Euler class of $E$. The Milnor--Wood inequality is an upper bound on the rotation number of a product of $g$ commutators in $\Homeo_+(S^1)$, or equivalently an upper bound on the Euler class of $E$. To say this in words more similar to \cref{thm:nofol2}, sufficiently small Dehn surgeries along a Seifert fiber result in a Seifert-fibered space without horizontal foliations.

\begin{figure}[h]
    \centering
    \def\svgwidth{4.5in}
    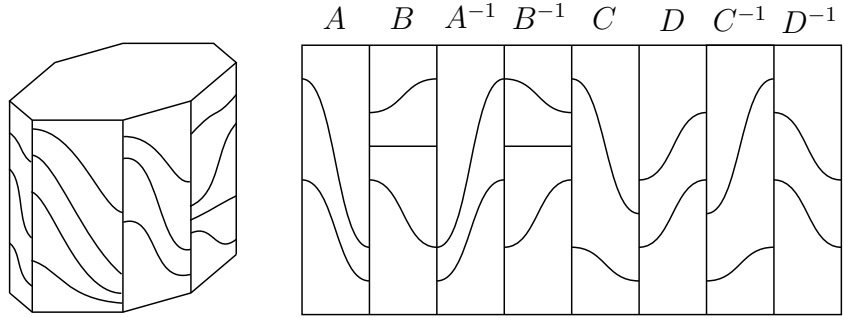
    \caption{$\partial D^2 \times S^1$ shown on the left. On the right, we unfolded $\partial D^2 \times S^1$ to show the induced foliation. The regions corresponding with $\rho(\gamma_1)$ and $\rho(\gamma_1)^{-1}$ are labelled with $A$ and $A^{-1}$.}
    \label{fig:milnorwood}
\end{figure}

Our analysis of foliations transverse to pseudo-Anosov flows follows a similar pattern. Instead of cutting along $f^{-1}(\cup_i \gamma_i)$, we cut along a branched surface which carries the weak stable foliation of the flow. The main technical difficulty lies in guaranteeing that the foliation is transverse to this branched surface. The flow is almost, but not quite tangent to the branched surface. Pictures analogous to the one on the right of \cref{fig:milnorwood} will appear in \cref{sec:examples}.

\end{subsection}

\subsection{Organization}
In \cref{sec:carried} we show that transverse foliations are carried by $\skel$. In \cref{sec:contact} we clarify the relationship between positive solutions to the holonomy problem and transverse contact structures and prove \cref{thm:converse}. In \cref{sec:examples}, we study the holonomy problem for 1-cusped veering triangulations in order to prove \cref{thm:nofol1} and \cref{thm:nofol2}. We conclude with some questions in \cref{sec:questions}
\subsection{Acknowledgements}
I would like to thank Thomas Massoni, Siddhi Krishna, and John Baldwin for explaining their work and many interesting discussions about applications and related questions. I would also like to thank Michael Landry and Sam Taylor for very helpful feedback, and in particular for explaining some of the ideas in \Cref{prop:carried}. Finally, I'm grateful for help from Saul Schleimer with the veering census.

\section{Preliminaries}
Throughout, let $M$ be a closed, oriented 3-manifold and let $\phi$ be a transitive pseudo-Anosov flow on $M$. We allow $\phi$ to have 1-prong singularities. Let $\gamma_1,\dots,\gamma_n$ be a finite collection of closed orbits of $\phi$ such that $\phi$ has no perfect fits relative to $\phi$. This set must contain the 1-prong singularities. Let $\Mpunc=M\setminus \cup_i \gamma_i$ and let $\Mbar$ be the manifold with torus boundary obtained by compactifying $\Mpunc$ with toroidal boundary components. The flow $\phi$ restricts to a flow $\phipunc$ on $\Mpunc$, and extends to a flow $\phibar$ on $\Mbar$ tangent to $\partial \Mbar$. The flow on $\partial \Mbar$ is not a linear flow, but it has a well-defined slope which is called the degeneracy slope.

In the case that $K$ is a knot in $S^3$ and $\phi$ is a pseudo-Anosov flow coming from a sutured hierarchy on $S^3\setminus K$, we can perform Dehn surgery along any slope not equal to the degeneracy slope to get a new pseudo-Anosov flow. We call it the \emph{natural pseudo-Anosov flow}. When $K$ is a fibered hyperbolic knot, the natural pseudo-Anosov flow always refers to the one transverse to the fibers.

The usual notion of orbit space needs some accommodations to deal with 1-prong singularities. For example, there are many 1-prong pseudo-Anosov flows on $S^3$, and we should not define the orbit space as a quotient of $S^3$. We replace $\widetilde M$ with $\widehat M$, the infinite order branched cover of $\lift M$ over the singular orbits in $\lift M$. We use widehats everywhere to denote the lifts of objects to $\widehat M$. Note that $\pi_1(\Mpunc)$ acts on $\widehat M$, and the quotient is $M$. Let $\ospace$ be the quotient of $\widehat M$ by $\widehat \phi$. We use $\pi$ to denote the projection from $\widehat M$ to $\ospace$. The action of $\pi_1(\Mpunc)$ descends to an action on $\ospace$. The weak stable and unstable foliations of $\widehat \phi$ descend to foliations of $\ospace$ with some infinite-prong singularities.

\subsection{Foliations of $T^2$}
A codimension 1 foliation of $T^2$ without Reeb annuli has a well-defined \emph{slope} in $\R \cup \infty$. It may be computed using the asymptotic slope of the lift of any leaf to the universal cover $\R^2$. We say that a folitaion has linear slope $s$ if it is isotopic to a linear foliation of slope $s$.

\subsection{Veering triangulations}
We recall some of the machinery of veering triangulations, but refer to \cite{landry.minsky.ea.FlowsGrowthRates} for a more thorough treatment. The manifold $\Mpunc$ admits a canonical ideal triangulation $\tau$. The tetrahedra are in one-to-one correspondence with maximal rectangles in the orbit space of $\phi$. The $i$-skeleton of $\tau$ is denoted $\tau^{(i)}$. The dual cell complex to $\tau$ is called $\tau^*$. Since $\tau$ has no vertices, $\tau^*$ has no 3-cells.

$\tau$ actually has the structure of a \emph{cooriented taut ideal triangulation}. This means that each face of $\tau$ has a coorientation (morally, the direction of the flow $\phi$ by \cite[Theorem 5.1]{landry.minsky.ea.FlowsGrowthRates}), and each tetrahedron has two faces with outward coorientation and two with inward coorientation. One usually thinks of the dihedral angle at an edge of a tetrahedron as either $0$ or $\pi$ so that the angle between two faces with the same coorientation is $\pi$, and 0 otherwise. This gives $\skel$ the structure of a cooriented branched surface without triple points.

One can check that the sum of the dihedral angles around any edge $e$ is $2\pi$. The two $\pi$ angles split the faces incident to $e$ into two groups called \emph{fans}. We use $n_e$ and $m_e$ to denote the number of triangles in these two fans. It doesn't matter much which side we choose to call $n_e$.

\subsection{The holonomy problem}\label{subsec:holonomyproblem}
 A \emph{candidate solution} for the holonomy problem is a choice of an orientation preserving homeomorphism $f_e: [0, n_e] \to [0,m_e]$ for each edge $e$ of $\tau$. A candidate solution specifes a lamination carried by $\skel$, formed as follows. Replace each face of $\tau$ with a trivially foliated $\Delta \times [0,1]$, where $\Delta$ is a triangle. At each edge $e$, glue the exposed faces of the $n_e$ thickened triangles in one fan of $e$ to the $m_e$ exposed faces in the other fan using $f_e$. See \cref{fig:triangles}. The complement of this lamination has one 3-ball per tetrahedron of $\tau$. These regions can be blown down to obtain a foliation on $\Mpunc$ transverse to $\phipunc$.

\begin{figure}[h]
    \centering
    \def\svgwidth{4in}
    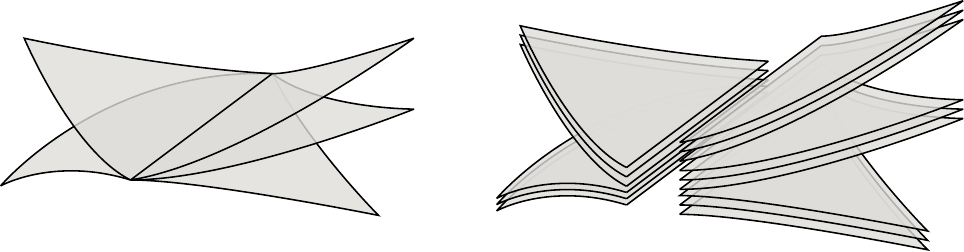
    \caption{An edge $e$ of $\skel$ with $n_e=2$ and $m_e=3$.}
    \label{fig:triangles}
\end{figure}

A candidate solution also gives rise to a foliation on $\Mpunc$ transverse to the dual cell complex $\tau^*$ as follows. If $e^*$ is the dual 2-cell to $e$, then $e^*$ is a polygon with $n_e + m_e$ sides. The candidate solution specifes up to isotopy a foliation on $e^*$ each of whose leaves starts one of the $n_e$ left sides and ends on one of the $m_e$ right sides, connecting a point $x\in [0,n_e]$ with a point $f_e(x)\in [0,m_e]$. See \cref{fig:dualtriangles}. This foliation extends to a codimension 1 foliation on a neighbourhood of the 2-skeleton of $\tau^*$. Note that this foliation is transverse to the 1-skeleton of $\tau^*$.

\begin{figure}[h]
    \centering
    \def\svgwidth{4in}
\begingroup%
  \makeatletter%
  \providecommand\color[2][]{%
    \errmessage{(Inkscape) Color is used for the text in Inkscape, but the package 'color.sty' is not loaded}%
    \renewcommand\color[2][]{}%
  }%
  \providecommand\transparent[1]{%
    \errmessage{(Inkscape) Transparency is used (non-zero) for the text in Inkscape, but the package 'transparent.sty' is not loaded}%
    \renewcommand\transparent[1]{}%
  }%
  \providecommand\rotatebox[2]{#2}%
  \newcommand*\fsize{\dimexpr\f@size pt\relax}%
  \newcommand*\lineheight[1]{\fontsize{\fsize}{#1\fsize}\selectfont}%
  \ifx\svgwidth\undefined%
    \setlength{\unitlength}{387.61307144bp}%
    \ifx\svgscale\undefined%
      \relax%
    \else%
      \setlength{\unitlength}{\unitlength * \real{\svgscale}}%
    \fi%
  \else%
    \setlength{\unitlength}{\svgwidth}%
  \fi%
  \global\let\svgwidth\undefined%
  \global\let\svgscale\undefined%
  \makeatother%
  \begin{picture}(1,0.30628262)%
    \lineheight{1}%
    \setlength\tabcolsep{0pt}%
    \put(0,0){\includegraphics[width=\unitlength,page=1]{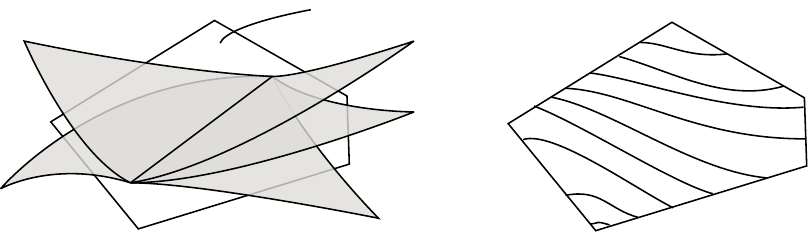}}%
    \put(0.3974516,0.28472134){\color[rgb]{0,0,0}\makebox(0,0)[lt]{\lineheight{1.25}\smash{\begin{tabular}[t]{l}$e^*$\end{tabular}}}}%
    \put(0,0){\includegraphics[width=\unitlength,page=2]{dualtriangles.pdf}}%
    \put(0.38855379,0.00417092){\color[rgb]{0,0,0}\makebox(0,0)[lt]{\lineheight{1.25}\smash{\begin{tabular}[t]{l}$e$\end{tabular}}}}%
    \put(0,0){\includegraphics[width=\unitlength,page=3]{dualtriangles.pdf}}%
  \end{picture}%
\endgroup%

    \caption{The foliation on the dual 2-cell $e^*$.}
    \label{fig:dualtriangles}
\end{figure}

\section{Transverse foliations are carried by $\skel$} \label{sec:carried}
 Let $\F$ be a codimension 1 foliation transverse to $\phi$. A \emph{rectangle} in $\ospace$ is a subspace of $\ospace$ homeomorphic to a product of intervals $I_1 \times I_2$ such that the induced stable and unstable foliations are horizontal and vertical respectively on $I_1 \times I_2$. Here, $I_1$ and $I_2$ are allowed to be open, closed, or half open intervals. A rectangle in $\widehat M$ is a disk in $\widehat M$ positively transverse to $\widehat \phi$ which projects to a rectangle in $\ospace$.

Let $\mathcal R$ be the set of maximal rectangles in $\ospace$. Let $\mathcal R_\F$ be the set of maximal rectangles in $\widehat M$ which are contained in leaves of $\widehat \F$. An \emph{$\F$-admissible lift} of a rectangle $r\subset \ospace$ is a rectangle in $\mathcal R_\F$ which projects via $\pi$ to $R$. 

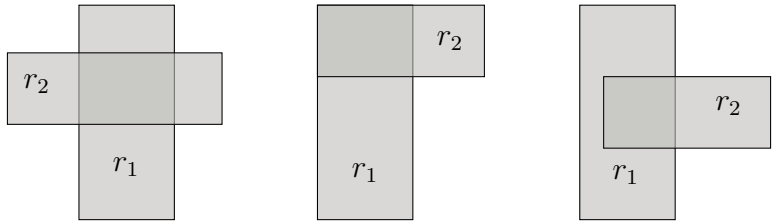
\begin{figure}[h]
    \centering
    \def\svgwidth{4in}
\begingroup%
  \makeatletter%
  \providecommand\color[2][]{%
    \errmessage{(Inkscape) Color is used for the text in Inkscape, but the package 'color.sty' is not loaded}%
    \renewcommand\color[2][]{}%
  }%
  \providecommand\transparent[1]{%
    \errmessage{(Inkscape) Transparency is used (non-zero) for the text in Inkscape, but the package 'transparent.sty' is not loaded}%
    \renewcommand\transparent[1]{}%
  }%
  \providecommand\rotatebox[2]{#2}%
  \newcommand*\fsize{\dimexpr\f@size pt\relax}%
  \newcommand*\lineheight[1]{\fontsize{\fsize}{#1\fsize}\selectfont}%
  \ifx\svgwidth\undefined%
    \setlength{\unitlength}{545.0019722bp}%
    \ifx\svgscale\undefined%
      \relax%
    \else%
      \setlength{\unitlength}{\unitlength * \real{\svgscale}}%
    \fi%
  \else%
    \setlength{\unitlength}{\svgwidth}%
  \fi%
  \global\let\svgwidth\undefined%
  \global\let\svgscale\undefined%
  \makeatother%
  \begin{picture}(1,0.28223907)%
    \lineheight{1}%
    \setlength\tabcolsep{0pt}%
    \put(0,0){\includegraphics[width=\unitlength,page=1]{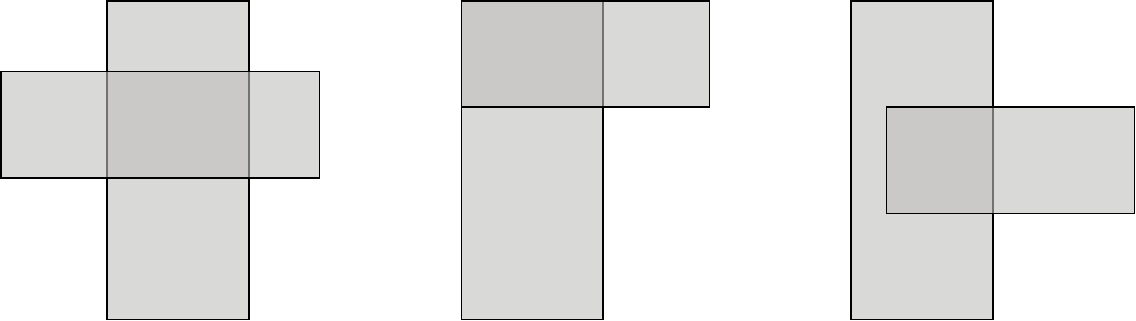}}%
    \put(0.15587668,0.06450018){\color[rgb]{0,0,0}\makebox(0,0)[t]{\lineheight{1.25}\smash{\begin{tabular}[t]{c}$r_1$\end{tabular}}}}%
    \put(0.03928711,0.17304158){\color[rgb]{0,0,0}\makebox(0,0)[t]{\lineheight{1.25}\smash{\begin{tabular}[t]{c}$r_2$\end{tabular}}}}%
    \put(0.46860098,0.05621733){\color[rgb]{0,0,0}\makebox(0,0)[t]{\lineheight{1.25}\smash{\begin{tabular}[t]{c}$r_1$\end{tabular}}}}%
    \put(0.57989677,0.22927225){\color[rgb]{0,0,0}\makebox(0,0)[t]{\lineheight{1.25}\smash{\begin{tabular}[t]{c}$r_2$\end{tabular}}}}%
    \put(0.80990166,0.05357135){\color[rgb]{0,0,0}\makebox(0,0)[t]{\lineheight{1.25}\smash{\begin{tabular}[t]{c}$r_1$\end{tabular}}}}%
    \put(0.94438358,0.14386542){\color[rgb]{0,0,0}\makebox(0,0)[t]{\lineheight{1.25}\smash{\begin{tabular}[t]{c}$r_2$\end{tabular}}}}%
  \end{picture}%
\endgroup%

    \caption{In the first two cases, $r_1 \leq_{\mathcal R} r_2$. In the last case, $r_1$ and $r_2$ are not comparable.}
    \label{fig:rectangles}
\end{figure}

There is a partial order $<_{\mathcal R}$ on $\mathcal R$ defined as follows. Given two rectangles $r_1,r_2 \subset \ospace$, we say $r_1 <_{\mathcal R} r_2$ if $r_1$ is at least as tall as $r_2$, $r_2$ is at least as wide as $r_1$, and $r_1\neq r_2$. See \cref{fig:rectangles}. There is also a partial order $<_{\F}$ defined on $\mathcal R_\F$. Given two rectangles $R_1,R_2\in \mathcal R_\F$, we say $R_1 <_\F R_2$ if there is positive transversal of nonzero length from the leaf containing $R_1$ to the leaf containing $R_2$. Tautness of $\widehat \F$ implies that $<_\F$ is a partial order.

A map $f:\RR \to \lR$ is called an \emph{$\F$-admissible lift} if 
\begin{enumerate}
    \item for every $r\in \RR$, $f(r)$ is an $\F$-admissible lift of $r$
    \item $f$ is $\pi_1(\Mpunc)$-equivariant
    \item $f$ is order preserving, ie $f(r_1) <_\F f(r_2)$ whenever $r_1 <_{\mathcal R} r_2$.
\end{enumerate}

\begin{lem}
    Let $I$ be an interval contained in a stable leaf in $\ospace$. Parameterize $\pi^{-1}(I)$ as $I\times \R$. Then the restriction of $\F$ to $I\times \R$ is a foliation with branching only in the negative direction. Similarly, if $I$ is an interval contained in an unstable leaf, then the restriction of $\F$ to $I\times \R$ is a foliation with branching only in the positive direction.
\end{lem}
\begin{proof}
    If $I\cong [0,1]$ is contained in a stable leaf, then flowlines in $I\times \R$ eventually stay within distance $\varepsilon$ of one another. The angle between $\F$ and flowlines of $\phi$ is bounded below. Therefore, for $z$ large enough, every leaf of $\widehat \F$ intersecting $0\times [z,\infty)$ also intersects $1\times \R$.
\end{proof}

\begin{lem}\label{lem:flowbox_structure}
    Given a closed rectangle $r\in \mathcal R$, parameterize $\pi^{-1}(r)$ as $r\times \R$ with coordinates $x,y,z$ where $x\in[0,1]$ parameterizes the stable direction, $y\in[0,1]$ parameterizes the unstable, and the flow is $(x,y,z) \mapsto (x,y,z+t)$. Let $\widehat \lambda$ be a leaf of $\widehat \F$ intersecting $r\times \R$. Thinking of $\widehat \lambda$ as the graph of a function $h:\pi(\widehat \lambda)\cap r \to \R$, at least one of the following holds:

    \begin{enumerate}
    \item $h$ is bounded below, and $\pi(\widehat \lambda) \cap r = I \times [0,1]$ for some interval $I \subseteq [0,1]$.
    \item $h$ is bounded above, and $\pi(\widehat \lambda) \cap r = [0,1] \times I$ for some interval $I\subseteq [0,1]$.
    \end{enumerate}
    Moreover, there is at least one leaf $\widehat \lambda$ for which both hold, and $\pi(\widehat \lambda) \supset r$. In other words, the foliation looks at worst like \cref{subfig:box1}.
    
\end{lem}
\begin{proof}
    Since $\F$ is taut, $\widehat \lambda$ is properly embedded in $\widehat M$. Therefore $\widehat \lambda$ is properly embedded in $r\times \R$. Note that $|\partial h/{\partial x}|$ is bounded on any region $z>const$, and $|\partial h / \partial y|$ is bounded on any region $z<const$. It follows that $h(x,y)$ cannot take both arbitrarily large and arbitrarily small values. For if $h(x_0,y_0)$ is very large, and $h(x_1,y_1)$ is very small, then $h(x_1,y_0)$ must be simultaneously large and small to satisfy the bounds on $\partial h/\partial x$ and $\partial h/\partial y$.

    Suppose $h(x,y)$ bounded below. Thanks to the upper bound on $\partial h/\partial x$, if $h(x_0,y_0)$ is finite, then so is $h(x,y_0)$ for any $x\in [0,1]$. Therefore, $\pi(\widehat \lambda)\cap r = [0,1]\times I$ for some interval $I \subset [0,1]$. If $I$ has an open endpoint, then $h(x,y)$ must diverge towards $+\infty$ as $y$ approaches this endpoint. If $h(x,y)$ is bounded below, then a similar argument shows that $\pi(\widehat \lambda)\cap r = I \times [0,1]$ for some interval $I \subset [0,1]$.

    Let $U^+$ be the set of leaves of $\widehat \F$ for which $h$ is bounded below, and let $U^-$ be the set of leaves which are bounded above. Both $U^+$ and $U^-$ are open in the leaf space of $\widehat \F \cap (r\times R)$, so $U^+ \cap U^-$ is nonempty. This proves the last assertion that there is a leaf $\widetilde \lambda$ for which $h$ is bounded above and below. 
\end{proof}

\begin{figure}[h!]
    \centering
    \subcaptionbox{\label{subfig:box1}}{
        \def\svgwidth{1.5in}
        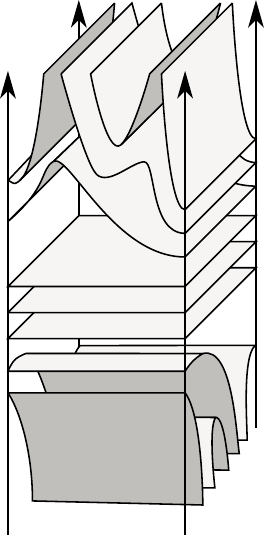
    }
    \hspace{0.4in}
    \subcaptionbox{\label{subfig:box2}}{
        \def\svgwidth{2.2in}
        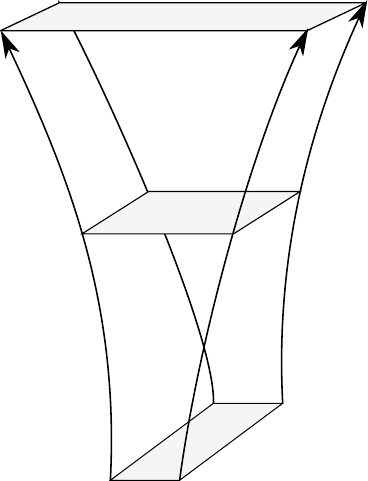
        }
    \caption{(a) shows the restriction of $\widehat \F$ to a flow cylinder, and (b) shows more realistic geometry of the same flow cylinder.}
    \label{fig:foliation_box}
\end{figure}

\begin{lem}\label{lem:nottoolow}
    Let $R$ be a rectangle in $\widehat M$ and let $r=\pi(R)$. Let $R'$ be another rectangle in $\widehat M$. Let $$\Omega = \bigcup_{\substack{g\in \pi_1(\Mpunc)\\ \pi(gR') >_{\mathcal R} \pi(R)}} gR' \cap \pi^{-1}(r).$$ Then $\Omega \subset r \times [C,\infty)$ for some constant $C$. Moreover, $\Omega$ contains only a finite number of rectangles in any compact subset of $\pi^{-1}(r)$.
\end{lem}
\begin{proof}
     Roughly speaking, the idea is that if a sequence of rectangles $g_i R'$ were to intersect our flow cylinder $\pi^{-1}(r)$ arbitrarily far backwards in the flow direction, the those rectangles would have to be very long in the stable direction. As above, parameterize the flow cylinder $\pi^{-1}(r)$ as $[0,1]\times[0,1]\times \R$. Suppose now that there is a sequence of elements $g_i\in \pi_1(M)$ such that $\pi(g_iR')>_{\mathcal R} r$ and $g_iR'$ intersects $[0,1]\times[0,1] \times (-\infty, -i]$. In particular, $g_iR'$ touches a point of the form $0\times y_0 \times N$ and $1 \times y_0 \times M$ for some $N,M < -i$. Now the following two facts are at odds with one another. First, the lengths of the leaves of the induced stable foliation on $g_i R'$ are uniformly bounded above since $g_i$ is an isometry. Second, let $\widehat \lambda$ be the weak stable leaf containing $0 \times y_0 \times \R$ and $1 \times y_0 \times \R$. Any weak stable leaf is Gromov-hyperbolic. Furthermore, the flowlines $0 \times y_0 \times \R$ and $1 \times y_0 \times \R$ are quasigeodesics in $\widehat \lambda$ with the same forward endpoint and different backward endpoints at infinity. Therefore, the distance between $0 \times y_0 \times (-\infty, -i]$ and $1 \times y_0 \times (-\infty, -i]$ in $\widehat \lambda$ diverges as $i\to \infty$. Contradiction.

     Only finitely many translates of $R'$ can intersect a compact subset of $\pi^{-1}(r)$ because a ball downstairs in $M$ which doesn't cross a singular orbit touches finitely many times the projection of $R'$ to $M$.
\end{proof}

\begin{lem}\label{lem:a}
    There exists an $\F$-admissible lift of $\mathcal R$.
\end{lem}
\begin{proof}
    Choose representatives $r_1\dots r_n$ for the $\pi_1(\Mpunc)$ orbits in $\mathcal R$. We will construct a lift $f$ by defining $f$ on $r_1,\dots,r_n$ one at a time. As noted in \cref{lem:flowbox_structure}, the restriction of $\widehat \F$ to $\pi^{-1}(r_1)$ has some leaf $\widehat \lambda$ such that $\pi(\widehat \lambda)\supseteq r_1$. Therefore, we can define $f(r_1)$ to be the lift of $r_1$ to $\widehat \lambda$. Extend $f$ to all the $\pi_1(\Mpunc)$-translates of $r_1$ equivariantly.

    Let us verify admissibility. Since $gr_1 > r_1$, there is a closed orbit representing $g$ from $r_1$ to $r_2$.
    
    If $g r_1 >_{\mathcal R} r_1$, then for all $n\geq 1$, $g^n r_1 >_{\mathcal R} r_1$. Suppose for the sake of contradiction that $f(gr_1) \leq_{\F} f(r_1)$. Then $g^n \, f(r_1) \leq_\F f(r_1)$ for all $n$. This violates \cref{lem:nottoolow}. Therefore, we must have had $f(gr_1) \geq f(r_1)$.

    Suppose we want to choose an $\F$-admissible lift for a rectangle $r_{k+1}$ having already chosen $\F$-admissible lifts for the orbits corresponding with $r_1,\dots,r_k$. Let $$\Omega_1 = \{f(gr_i) \mid g\in \pi_1(M),\, 1\leq i \leq k,\, gr_i >_{\mathcal R} r_{k+1} \}$$ and
    $$\Omega_2 = \{f(gr_i) \mid g\in \pi_1(M),\, 1\leq i \leq k,\, gr_i <_{\mathcal R} r_{k+1} \}$$

Both $\Omega_1$ and $\Omega_2$ are non-empty because there are translates of $r_1$ which are greater or less than $r_{k+1}$. Since $f$ is an $\F$-admissible lift of $\{r_1,\dots,r_k\}$, we know that $R_1 >_\F R_2$ for any $R_1\in \Omega_1, R_2 \in \Omega_2$. Let $I$ be the interval in the leaf space of $\widehat \F \cap \pi^{-1}(r_{k+1})$ consisting of leaves whose $\pi$ image covers $r_{k+1}$. This interval is nonempty by \cref{lem:flowbox_structure}. Let $\omega_1,\omega_2$ be the subsets of $I$ consisting of leaves containing rectangles in $\Omega_1$, $\Omega_2$ respectively. By the structure theorem in \cref{lem:flowbox_structure}, $\omega_1$ does not contain any leaf below $I$ and $\omega_2$ does not contain any leaf above $I$. Since our lift is $\F$ admissible so far, that $\lambda_1 >_{\F} \lambda_2$ for any $\lambda_1 \in \omega_1, \lambda_2 \in \omega_2$. By \cref{lem:nottoolow}, $\inf(\omega_1)$ exists (or $\omega_1$ is empty) and similarly $\sup(\omega_2)$ exists (or $\omega_2$ is empty). Moreover, by the same lemma, the infimum and supremum are realized. Therefore, there is some element $\lambda \in I$ satisfying $$\omega_1 >_\F \lambda >_\F \omega_2.$$ We choose $f(r_{k+1})$ to be $\lambda$. As in the argument for $r_1$, the admissibility condition holds for all pairs of translates of $r_{k+1}$. By construction, admissibility holds for all other pairs of rectangles.
\end{proof}

\begin{lem}\label{lem:b}
    Given an $\F$-admissible lift $f:\RR \to \lR$, $\F^\circ$ can be split open to a lamination fully carried by $\skel$.
\end{lem}
\begin{proof}
    By \cite[Theorem 5.1]{landry.minsky.ea.FlowsGrowthRates}, the veering triangulation $\tau$ may be realized in $\Mpunc$ with 2-skeleton positively transverse to the flow. For each rectangle $r\subset \ospace$, let $T_r\subset \ospace$ be the shadow of the corresponding tetrahedron of $\widehat \tau$. \cite{landry.minsky.ea.FlowsGrowthRates} arranges that $\pi(T_r)$ is a quadrilateral inscribed in $r$. Delete the singular orbits of $\widehat M$, then cut open $\widehat \F$ along $f(\pi(T_r))$ and blow air into the incision. This is the inverse of the blowdown operation described in \cref{subsec:holonomyproblem}. After projecting to $\Mpunc$, we claim that the resulting lamination is isotopic to one carried by $\skel$. We describe the isotopy more explicitly below.
    
    Let $\gamma$ be an orbit of $\lift \phipunc$. It intersects the interiors of bi-infinite sequence of edges or faces of $\lift \skel$. Call this sequence $\dots, x_{-1},x_0,x_1,x_2,\dots$. Define a map $g_{\gamma}:\gamma \to \bigcup_i (\gamma \cap x_i)$ as follows. For each $i$, let $S_i$ be the set of rectangles in the image of $f$ corresponding with a tetrahedron incident to $x_i$. See \cref{fig:lift}. Then $g_\gamma$ maps the half-open interval $(\min(S_i)\cap \gamma, \max(S_i)\cap \gamma]$ to $\gamma \cap x_i$. Since $f$ is strictly order preserving, $g_\gamma$ is monotonic and surjective. Also, the points of discontinuity of $g_\gamma$ are precisely the points of intersection between $\gamma$ and $\bigcup_r f(\pi(T_r))$.
    
    Now assemble the various $g_\gamma$ together into a big map $g:\lift \Mpunc \to \lift \Mpunc$. This map preserves orbits of $\lift \phi$, is monotonic on each orbit, and is continuous except at the incisions we made. Now linear interpolation between $g$ and the identity is an isotopy which moves the foliated region onto a neighbourhood of $\lift \skel$. The whole construction is $\pi_1(\Mpunc)$ equivariant, so it descends to an isotopy in $\Mpunc$ which moves the split open foliation onto a neighbourhood of $\skel$. The surjectivity of $g$ justifies our claim that the lamination is fully carried.

    \begin{figure}[h!]
        \centering
        \def\svgwidth{4in}
        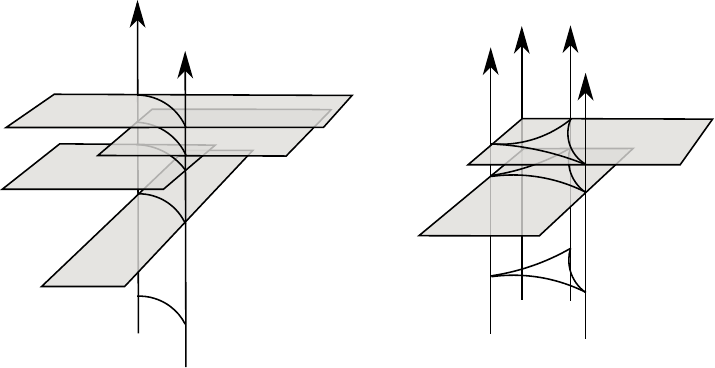
        \caption{On the left, we have shown $S_i$ when $x_i$ is an edge. In general, $|S_i|$ depends on the sizes of the fans at $x_i$. On the right is $S_i$ when $x_i$ is a triangle. In this case, $|S_i|=2$.}
        \label{fig:lift}
    \end{figure}

\end{proof}

\cref{prop:carried} now follows from \cref{lem:a} and \cref{lem:b}.

\begin{rmk}
    It is not easy to apply Li's theory of laminar branched surfaces to $\skel$ because $\skel$ is not in generic position; it has many sheets which come together along a single branch locus and it is not clear which way one should split these apart. On the other hand, because $\skel$ has no triple points, we will see that it is sometimes easier to understand which boundary slopes are carried.
\end{rmk}

\begin{section}{From positive solutions to contact structures}\label{sec:contact}
\begin{lem}\label{lem:contactexist}
    If there is a positive (resp negative) solution to the holonomy problem, then there exists a positive (resp negative) contact structure transverse to $\phi$.
\end{lem}
\begin{proof}
    A positive solution to the holonomy problem gives rise to a foliation $\F^+$ on $\Mbar$ transverse to $\phibar$ with positive slope on $\partial \Mbar$. Let us focus on one of the boundary components of $\Mbar$. We would first like to arrange that $\F^+$ has a standard form in a neighbourhood of this boundary component. Parameterize a small neighbourhood of this boundary component as $T^2 \times [0,\varepsilon)$ using variables $\varphi, \theta, r$. Here, $\varphi$ parameterizes the flow direction and $\theta$ parameterizes the meridional direction. We will put $\F^+$ in standard form using 2D mean curvature flow. Endow $T^2$ with its standard product Euclidean metric. Recall the following facts about mean curvature flow:
    \begin{enumerate}
        \item A foliation remains a foliation for all time
        \item Mean curvature flow on a foliation of slope $s$ approaches a linear slope $s$ foliation in the $C^1$ topology.
        \item If two foliations are transverse, then they remain transverse for all time.
    \end{enumerate}

    Now $T^2 \times 0$ has two transverse foliations, one by flowlines of $\phi$ and the other by the restriction of $\F$. Apply mean curvature flow to both. The first foliation has slope $\infty$ and the second has positive slope. Therefore, after some large time, we can guarantee that the two foliations are transverse and as close to linear foliations as desired. In particular, we can choose a large enough time that the slope of $\F|_{T^2}$ is pointwise positive. Call the two isotopies $f_t: T^2 \to T^2$ and $g_t: T^2 \to T^2$, where $t$ ranges over $[0,1]$. We can extend these isotopies to isotopies of $\Mbar$ supported on $T^2 \times [0,\varepsilon]$ by the formula $(\varphi,\theta,r) \mapsto (f_{1-r/\varepsilon}(\varphi,\theta), r)$ and $(\varphi,\theta,r) \mapsto (g_{1-r/\varepsilon}(\varphi,\theta), r)$. For small enough $\varepsilon$, $\F^+$ and $\phi$ remain transverse during this isotopy. 

    As $r\to 0$, the component of $\phi$ in the $r$ direction uniformly goes to zero. Isotope $\F^+$ so that $T\F$ is tangent to the $r$ direction in a neighbourhood of $T^2 \times 0$ and $\F^+$ remains transverse to $\phi$. Now we perturb $T\F^+$ to a confoliation by decreasing the slope of $T\F^+$ down to zero as $r\to 0$. Finally, reparameterize $\phi$ so that it has unit speed in the $\varphi$ direction. Then we may blow down $T^2 \times 0$ to recover a flow which is a reparameterization of $\phi$, along with a transverse confoliation. Finally, apply the Eliashberg--Thurston theorem to perturb the confoliation to a contact structure transverse to $\phi$ \cite{eliashberg.thurston.Confoliations}. Note that this is the easy case of the Eliashberg--Thurston, where every leaf of the confoliation touches the contact region. On the other hand, the confoliation is not necessarily transversely smooth, so it is safest to use the $C^0$ version due to Kazez and Roberts or Bowden \cite{kazez.roberts.C^0ApproximationsFoliations,bowden.Approximating0foliationsContact}.
\end{proof}

\begin{prop}\label{prop:forward}
    If there is a foliation $\F$ transverse to $\phi$, then there is a positive and a negative solution to the holonomy problem.
\end{prop}
\begin{proof}
    By \cref{prop:carried}, $\F$ is carried by $\skel$. Therefore, there is a candidate solution to the holonomy problem with linear zero slope around each cusp of $\Mpunc$. As in \cref{fig:dualtriangles}, we may use this candidate solution to build a foliation of $\tau^*$. Since the foliation has linear zero slope around each cusp, it extends to foliation $\F'$ of $M$ which is isotopic to $\F$, transverse to $\tau^*$, and transverse to the 1-skeleton of $\tau^*$. Now apply the Eliashberg--Thurston theorem to get a pair of tight contact structures $\xi_+$ and $\xi_-$ which $C^0$ approximate $\F'$. These contact structures can be chosen $C^0$ close to $\F'$, so they are also transverse to the 1-skeleton and 2-skeleton of $\tau^*$. Therefore, $\xi_+ \cap \tau^*$ and $\xi_-\cap \tau^*$ give rise to candidate solutions to the holonomy problem. We claim that these solutions are positive and negative respectively. Indeed, suppose $\xi_+ \cap \tau^*$ gives rise to a solution which is not positive at the $i^{th}$ cusp. Then there is a closed curve $K$ in $\tau^*$ negatively transverse to $\xi^+$ bounding a disk in $M \setminus \tau^*$ which intersects $\gamma_i$ positively. Then $sl(K)=1$, violating the Bennequin inequality which in this case says $sl(K)\leq -1$.
\end{proof}

Now we recall Massoni's converse of the Eliashberg--Thurston theorem:

\begin{thm}[{\cite{massoni.TautFoliationsContact}}]\label{thm:massoni}
    If $\xi_-$ and $\xi_+$ are a pair of positive and negative tight contact structures transverse to a common flow $\phi$, then there exists a $C^0$ foliation transverse to the same flow.
\end{thm}

\begin{prop}\label{prop:backward}
    If there is a positive and a negative solution to the holonomy problem, then there is a foliation transverse to $\phi$.
\end{prop}
\begin{proof}
    By \cref{lem:contactexist}, there is a pair of contact structures $\xi_+$ and $\xi_-$ transverse to $\phi$. First we verify that $\xi_+$ and $\xi_-$ are tight. For pseudo-Anosov flows without 1-prongs (which in particular have no contractible orbits), this is argued in \cite[Proposition 5.1]{zung.PseudoAnosovRepresentativesStable}. Let's give a slightly different argument that works even in the 1-prong setting. We thank Thomas Massoni for pointing out this argument. First, we find a $C^0$ approximation of the vector field generating $\phi$ by a smooth divergence free vector field $\phi'$. One way to do this is the blowup procedure described in \cite[Section 3]{zung.PseudoAnosovRepresentativesStable}, which works equally well in the 1-pronged case. Now, following Eliashberg--Thurston, we construct a weak symplectic filling of $(M,\xi_+) \cup (-M,\xi_-)$. Let $\omega$ be the closed 2-form dual to $\phi'$, and let $\alpha$ be a contact form whose kernel is $\xi_+$. Use $t$ as a parameter for $[-\varepsilon, \varepsilon]$. Then $\omega + d(t\alpha)$ is a symplectic form on $[-1,1]\times M$. For small enough $\varepsilon$, the restriction of the symplectic form to the boundary is positive on $\xi_+$ and $\xi_-$. Therefore, $[-\varepsilon,\varepsilon]\times M$ is a weak symplectic filling of $(M,\xi_+) \cup (-M, \xi_-)$ and $\xi_+$ and $\xi_-$ must be tight. Now we may apply \cref{thm:massoni} to obtain a foliation transverse to $\phi$.
\end{proof}

\Cref{prop:forward} and \cref{prop:backward} are the two halves of \Cref{thm:converse}.

\end{section}

\begin{section}{Criteria for carriage}\label{sec:examples}
Our veering triangulation induces a triangulation $L$ of $\partial \Mbar$. In this section, we focus on examples where $\partial \overline{\Mpunc}$ has only one component. We draw all our pictures with the degeneracy slope vertical. The triangulation is divided into a number of \emph{ladders} parallel to the degeneracy slope. Each ladder is a triangulated annulus with vertices only its boundary. The boundaries of the ladders are called \emph{ladderpoles}, and the interior edges are called \emph{rungs}. Since $\tau$ is a branched surface, $L$ is a train track. The ladderpoles are all carried by this train track. Ladderpoles alternate between \emph{upward} and \emph{downward}. In an upward ladderpole, the rungs on the left are smoothed upward and the rungs on the right are smoothed downward. The opposite is true for downward ladderpoles. Thus, as one moves from left to right along a curve carried by $L$, one ascends along upward ladderpoles and descends along downward ladderpoles. An important property for us is that the two endpoints of an edge of $\tau$ either both lie on descending ladderpoles or both lie on ascending ladderpoles.

\begin{figure}[h]
\centering
\includegraphics[width=3.5in]{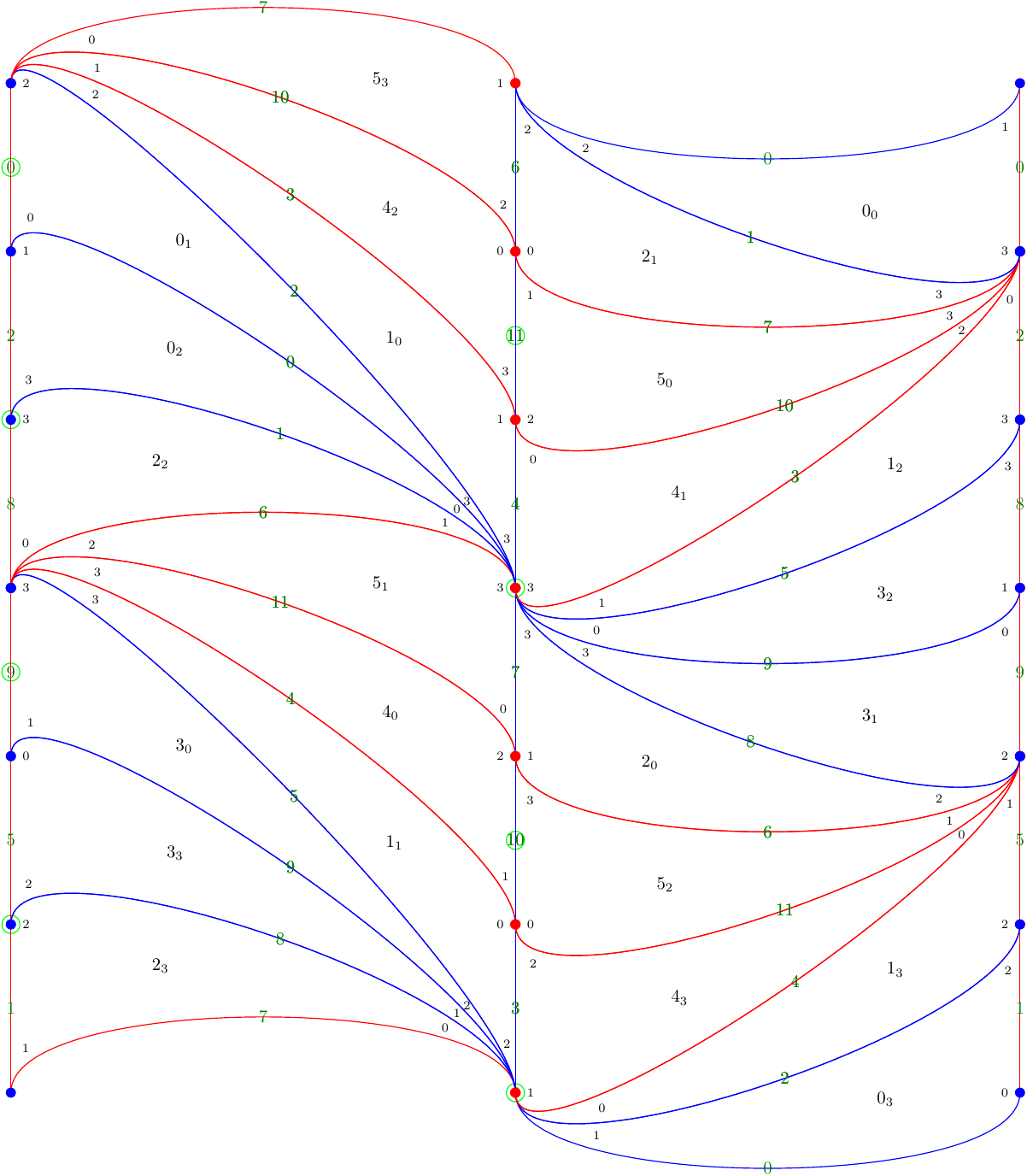}
\caption{The boundary train track for $10_{145}$, reproduced from the veering census \cite{parlak.schleimer.ea.Veering03Code}}\label{fig:10_145}
\end{figure}

Let us study the example of the complement of the knot $K=10_{145}$. Let $\Mpunc=S^3\setminus K$. The train track on $\partial \Mbar$ is shown in \cref{fig:10_145}. There is one upward ladderpole and one downward ladderpole. The left and right sides of the picture are identified, and the top and bottom are also identified. Edges in this diagram are labelled with the index of the corresponding face of $\tau$, and vertices are labelled with the index of the corresponding edge of $\tau$. Recall that a foliation carried by $\skel$ is specified by a choice of a homeomorphism $f_i: [0,n_i] \to [0,m_i]$ for each edge $i$. Here, $n_i$ and $m_i$ represent the number of triangles in either fan around edge $i$.

In this section, we will work with slopes in the coordinate system defined by this picture. At the end we will convert back into the standard slopes on $S^3$. \cref{fig:10_145_slopes} shows some curves which may be realized as leaves of solutions to the holonomy problem. The first is a curve of slope $\leq -1/3$. The second is the longitude, which has slope $-1/6$. The last is the meridian, which has slope $0$. The first curve passes through the edge labelled $A/A^{-1}$ multiple times, so some care must be taken to ensure that the claimed curve is actually realizable. Each fan adjacent to this edge has 5 triangles. Define $A:[0,5]\to [0,5]$ so that
\begin{align*}
    A(1.5) = 1.4\\
    A(3.4) = 1.6\\
    A(3.6) = 2.5
\end{align*}

In \cref{subfig:negative_slope}, the branches incident to $A$ or $A^{-1}$ are labelled with their proposed heights in $[0,5]$. Thus, we have proven:

\begin{lem}\label{lem:slopes1}
   Slopes -1/3 and 0 are realized by solutions to the holonomy problem in \cref{fig:10_145}.
\end{lem}

\begin{figure}[!h]
    \centering
    \subcaptionbox{\label{subfig:negative_slope}}{
    \def\svgwidth{2.75in}
    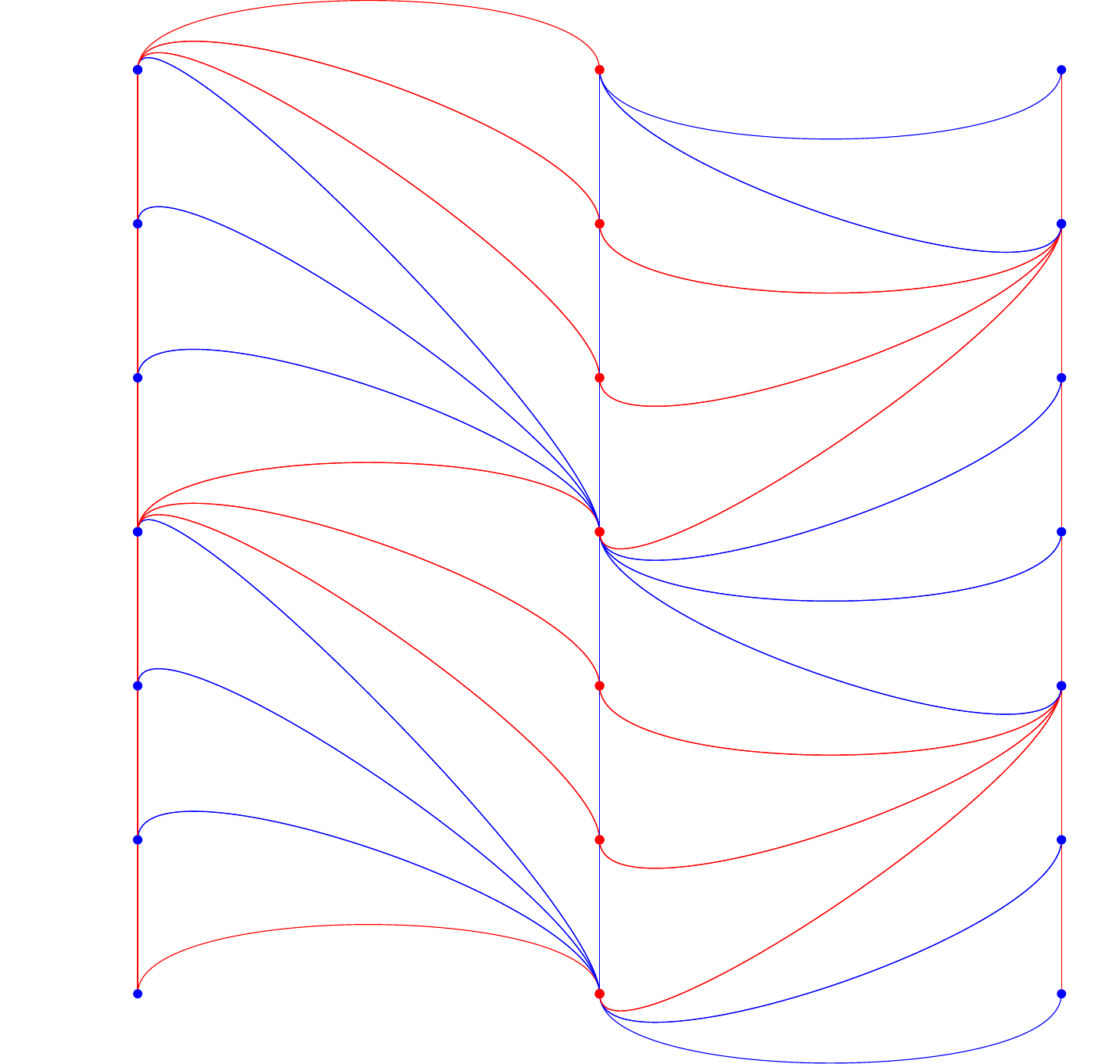
    }
    \subcaptionbox{\label{subfig:longitude}}{
    \def\svgwidth{2.75in}
    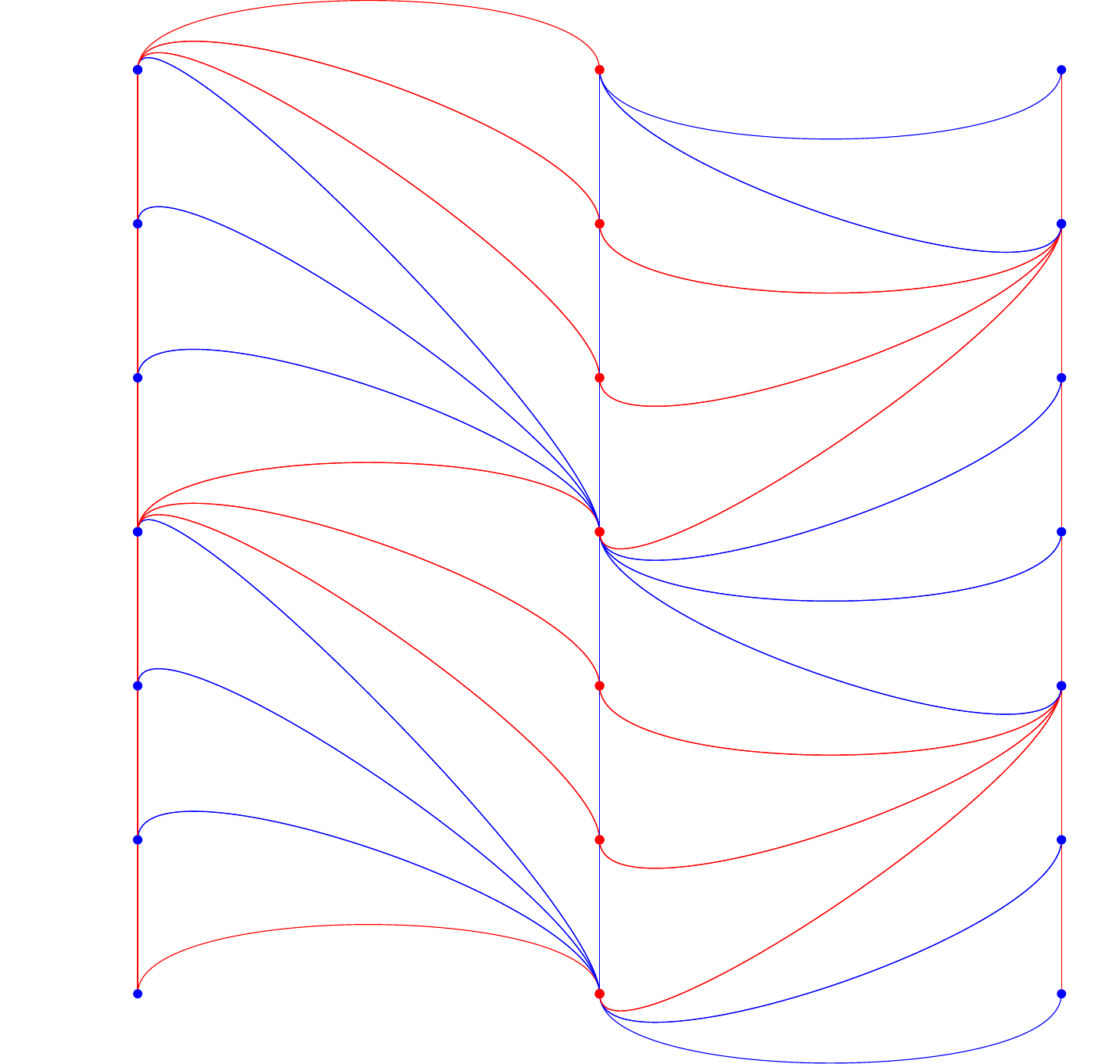
    }
    \subcaptionbox{\label{subfig:meridian}}{
        \def\svgwidth{2.75in}
        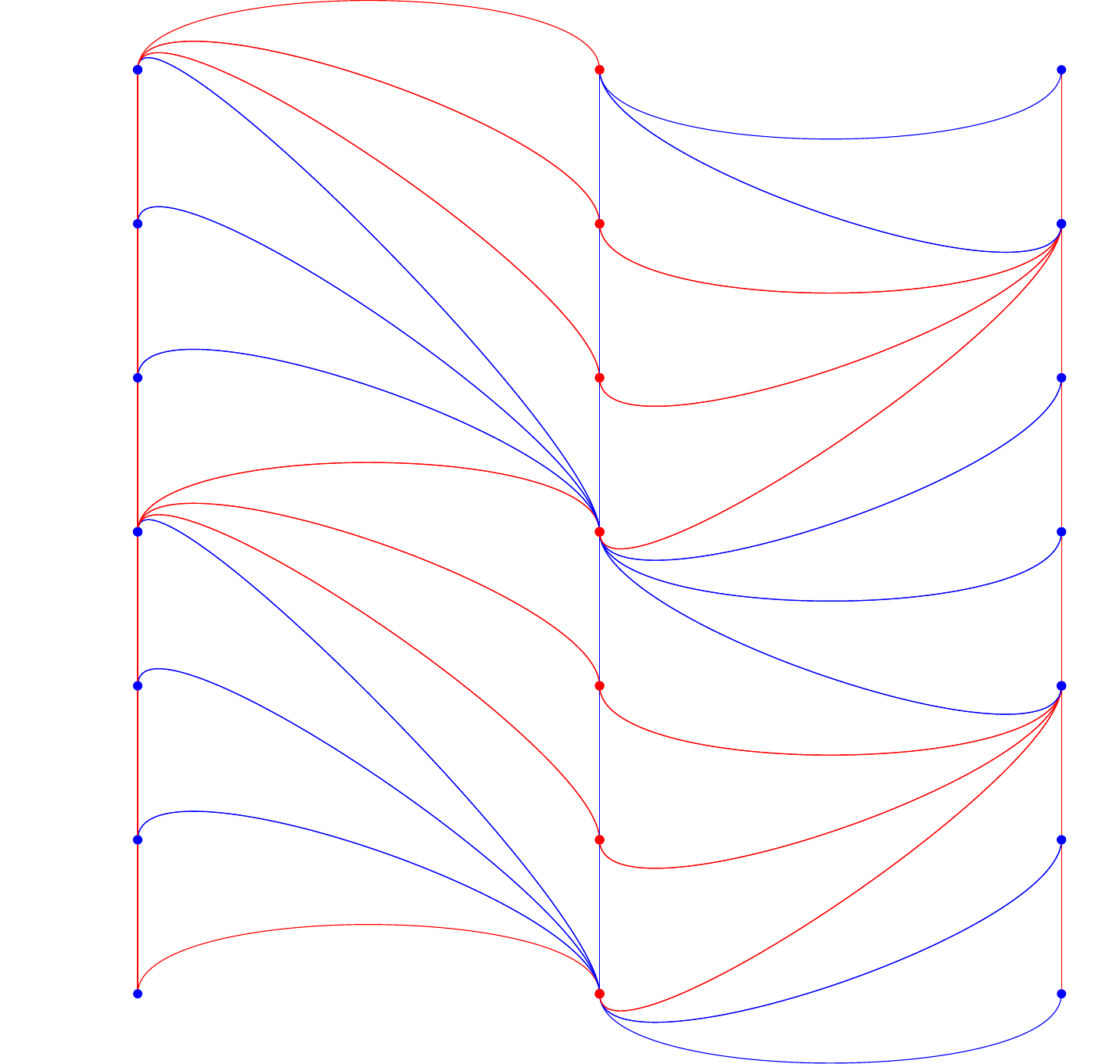
        }
    \caption{Shown in black are some curves that can be realized as integral curves of solutions to the holonomy problem. \ref{subfig:negative_slope} shows a curve of slope $\leq -1/3$, while \ref{subfig:longitude} shows the longitude which has slope $-1/6$. The integer weights in \ref{subfig:longitude} show the multiplicity with which the curve crosses that branch of the train track. The weights in \ref{subfig:negative_slope} show the height of the curve in the fan.}\label{fig:10_145_slopes}
\end{figure}

In contrast, we have:

\begin{lem}\label{lem:slopes2}
    There is no positive solution to the holonomy problem in \cref{fig:10_145}.
\end{lem}
\begin{proof}
    We already showed that there is a negative solution to the holonomy problem. If there were a positive solution to the holonomy problem, then we could apply \cref{thm:converse} to conclude that there is a foliation on $S^3$ transverse to a transitive 1-prong pseudo-Anosov flow. This would be a taut foliation on $S^3$. Contradiction.
\end{proof}

\begin{lem}\label{lem:slopes3}
    There is no solution to the holonomy problem with slope less than -1.
\end{lem}
\begin{proof}
    Consider an integral curve passing through the rung labelled $p$. It travels down through the middle ladderpole for some time, and then exits along some rung on the right ladder. See \cref{fig:10_145_impossible}. Note that during its descent along the downward ladderpole, the integral curve cannot pass through both the edge marked $B^{-1}$ and the edge marked $B$. This is because the homeomorphism $B:[0,2]\to [0,2]$ would have to send a point in $[1,2]$ to a point in $[0,1]$ as well as a point in $[0,1]$ to $[1,2]$. Therefore, the lowest rung at which this curve can leave the middle ladderpole is the rung marked $q$. This curve cannot return to the left ladder with slope $<-1$.

    \begin{figure}[!h]
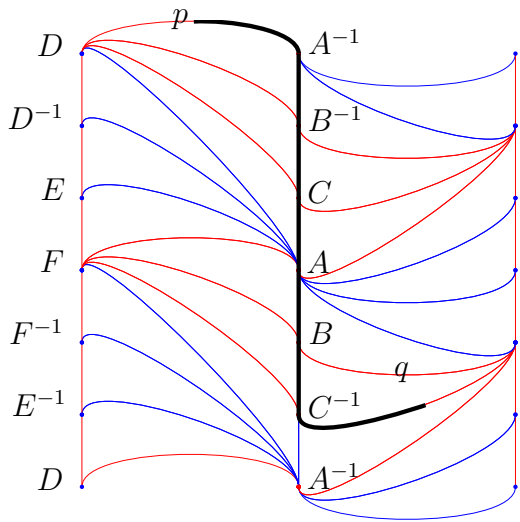

        \centering
        \def\svgwidth{2.75in}
        
        \caption{A curve that is not realizable.}\label{fig:10_145_impossible}
    \end{figure}
\end{proof}

\begin{proof}[Proof of \cref{thm:nofol1}]
Combining \cref{lem:slopes1} with \cref{thm:converse}, we know that Dehn surgery along slopes in $(-1/3,0)$ results in a pseudo-Anosov flow with a transverse foliation. Combining \cref{lem:slopes2}, \cref{lem:slopes3}, and \cref{thm:converse}, we know that Dehn surgery along slopes in $(0,\infty)$ and $(-\infty,-1)$ result in a pseudo-Anosov flow without a transverse foliation.

Writing these slopes in terms of the longitude $\lambda = (6,-1)$ and meridian $\mu = (-1,0)$ gives the statement claimed in \cref{thm:nofol1}. We have $(3,-1) = \lambda + 3\mu$ and $(1,-1) = \lambda + 5\mu$. Note that the orientation of our pictures is opposite to the standard orientation $(\lambda, \mu)$; we thank Siddhi Krishna for pointing this out.
\end{proof}










Having completed this example, we can now generalize the argument to 1-cusped veering triangulations with only one upward and one downward ladder. 
\Cref{thm:nofol2} follows immediately from the following observation:

\begin{prop}
    Suppose $\tau$ is a 1-cusped veering triangulation whose boundary triangulation has only one upward ladder and one downward ladder. Then solutions to the holonomy problem have slope between $-4$ and $4$.
\end{prop}
\begin{proof}
    Suppose we have an integral curve with slope $\leq -4$. The curve descends by at most 1 in the first ladder and at most 1 in the second ladder. Therefore, it must descend by at least 2 along the descending ladderpole. But then during this descent, it must cross two endpoints of the same edge of $\tau$, just as the curve in \cref{fig:10_145_impossible} crosses both $B$ and $B^{-1}$. As argued previously, there is no choice of homeomorphism at this edge allowing for this integral curve. Contradiction. A similar argument rules out slopes $\geq 4$.
\end{proof}

\begin{rmk}
	We found the solution in \cref{lem:slopes1} by hand, but is also fairly easy to find solutions with a computer search. Indeed, if there is a positive solution to the holonomy problem, then there is a piecewise linear positive solution. We implemented a search for piecewise linear solutions using simulated annealing. In practice were often able to find the maximal interval of slopes with transverse foliations for small examples of knot complements in the census of veering triangulations. We welcome suggestions of interesting examples to investigate.
\end{rmk}

\end{section}




\begin{section}{Questions}\label{sec:questions}
\begin{question}
    If $\phi$ admits a transverse contact structure, is there a positive solution to the holonomy problem? In other words, are all transverse contact structures carried in a suitable sense by $\skel$? 
\end{question}
\begin{question}
    A more conscientious mathematician would attempt to replace transversality with almost-transversality in the sense of Mosher. How does the author sleep at night?
\end{question}
\begin{question}\label{ques:uniquecontact}
 If $\phi$ admits a transverse contact structure, is it unique up to isotopy?
\end{question}
If this were true, then it would prove uniqueness of the Eliashberg--Thurston perturbation of any foliation admitting a transverse pseudo-Anosov flow. Vogel previously proved uniqueness for $C^2$ taut foliations on atoroidal manifolds. One way to approach \cref{ques:uniquecontact} would be to first show that the set of positive solutions to the holonomy problem is connected.
\end{section}

\printbibliography
\end{document}